\documentclass[
12pt]{amsart}
\usepackage{amscd}
\usepackage{amssymb}
\usepackage{a4wide}
\usepackage{amstext}
\usepackage{amsthm}
\usepackage{mathrsfs}
\usepackage{ulem}
\usepackage{xcolor}
\usepackage[final]{hyperref}
\usepackage{cancel}
\numberwithin{equation}{section}

\newcommand{\half}{\frac{1}{2}}

\DeclareMathOperator{\inter}{int}

\renewcommand{\phi}{\varphi}

\newcommand{\pw}{\mathcal{P}W_\pi}

\newcommand{\la}{\lambda}
\newcommand{\La}{\Lambda}

\newcommand{\lng}{\langle}
\newcommand{\rng}{\rangle}

\newcommand{\R}{{\Bbb R}}

\newcommand{\Cp}{{\Bbb C}_+}
\newcommand{\Cm}{{\Bbb C}_-}
\newcommand{\Cpm}{{\Bbb C}_\pm}

\newcounter{tmp}

\newcommand{\closspan}{\overline{{\rm span}}}

\newcommand{\conv}{conv}
\theoremstyle{plain}
\newtheorem{theorem}{Theorem}[section]

\newtheorem{proposition}{Proposition}[section]
\newtheorem{lemma}{Lemma}[section]

\newtheorem{definition}{Definition}[section]

\newtheorem{remark}{Remark}[section]

\begin{document}  
\title[On summation of non-harmonic Fourier series]{On summation of non-harmonic Fourier series}
\author{ Yurii Belov, Yurii Lyubarskii}

\thanks{The first author was supported by the Chebyshev Laboratory 
(St. Petersburg State University) under RF Government grant 11.G34.31.0026,
by JSC "Gazprom Neft", and  by RFBR grants 12-01-31492 and 14-01-31163. The second author is partly supported by the Norwegian Research Council project DIMMA \#213638.
Part of this work was done while the authors were staying at the Center for Advanced Study, Norwegian Academy of Science, and they would like to express their gratitude to the institute for the hospitality. }

\address{Yurii Belov,
\newline 
Chebyshev Laboratory,
St.Petersburg State University,
St.Petersburg, Russia,
\newline {\tt j\_b\_juri\_belov@mail.ru}
\newline\newline \phantom{x}\,\, 
Yurii Lyubarskii, 
\newline
Dept. Mathematical Sciences, Norwegian University of Science and
Technology,NO-7491 Trondheim, Norway,
\newline {\tt yura@math.ntnu.no}
}

\begin{abstract}
Let a sequence $\Lambda\subset\mathbb{C}$ be such that the corresponding system of exponential functions $\mathcal{E}(\Lambda):=\{e^{i\lambda t}\}_{\lambda\in\Lambda}$ is complete and minimal in $L^2(-\pi,\pi)$
and thus each function $f\in L^2(-\pi,\pi)$ corresponds to a non-harmonic Fourier series in $\mathcal{E}(\Lambda)$. 
We prove that if the generating function $G$ of $\Lambda$ satisfies Muckenhoupt $(A_2)$ condition on $\mathbb{R}$, then this series admits a linear summation method.
 Recent results  show that $(A_2)$ condition cannot be omitted.
\end{abstract}


\maketitle
 
 
\section{Introduction and main results}

Let a sequence $\Lambda\subset\mathbb{C}$ be such that the corresponding exponential system $\mathcal{E}(\Lambda):=\{e^{i\lambda t}\}_{\lambda\in\Lambda}$ is complete and minimal in $L^2(-\pi,\pi)$. By $\{\phi_\lambda\}_{\lambda\in\Lambda}\subset L^2(-\pi,\pi)$ we denote the biorthogonal system: $(e^{i\mu t}, \phi_\lambda)=\delta_{\lambda,\mu}$, $\lambda,\mu\in\Lambda$. To each function $f\in L^2(-\pi,\pi)$  one can associate  its non-harmonic Fourier series in $\mathcal{E}(\Lambda)$:
\begin{equation}
f \sim \sum_{\lambda\in\Lambda}(f,\phi_\lambda)e^{i\lambda t}.
\label{mainS}
\end{equation}
 It was shown  by R.Young in \cite{Young} that the biorthogonal system $\{\phi_{\lambda}\}_{\lambda\in\Lambda}$ is also complete, hence, the Fourier coefficients $(f,\phi_\lambda)$ determine the
function $f$  uniquely. In this article we study reconstruction of $f$ from series \eqref{mainS} by a linear summation method.
\begin{definition}
Given a matrix $W=\{w(\lambda,n)\}_{\lambda\in\Lambda, n\in\mathbb{N}}$ with 
         \[
\lim_{n\rightarrow\infty}w(\lambda,n)=1, \quad \lim_{\lambda\rightarrow\infty}w(\lambda,n)=0,
\]    
we say that it provides a linear summation method for series \eqref{mainS} if, for any $f\in L^2(-\pi,\pi)$, the series
\begin{equation}
S_n(W,f):=\sum_{\lambda\in\Lambda}w(\lambda,n)(f,\phi_\lambda)e^{i\lambda t} 
\label{summ}
\end{equation}
 converges in $L^2(-\pi,\pi)$ and 
\[
S_n(W,f) \to f, \ \text{as} \ n\to \infty \ \text{in} \ L^2(-\pi,\pi).
\]
 \end{definition} 

It follows from the completeness and minimality of  $\mathcal{E}(\Lambda)$ (see, e.g. \cite[Lecture 17]{Lev}) that 
 there exists the generating function,
$$G(z):=G(0)\lim_{R\rightarrow\infty}\prod_{\lambda\in\Lambda,|\lambda|<R}\biggl{(}1-\frac{z}{\lambda}\biggr{)},$$
and that $G$ is of exponential type $\pi$ in both half-planes $\mathbb{C}_\pm$.

The summability properties of series \eqref{mainS} can be described through this function.
  In this introduction we assume that, for some $\delta>0$, 
all points in $\Lambda$ belong to a shifted upper half-plane, $\Lambda\subset\mathbb{C}_\delta:=\{\zeta: \Im\zeta>\delta\}$.  This is done in order to formulate the results in the introduction as simple as possible. Later in  the article we present the results in full generality, i.e. not assuming $\Lambda\subset\mathbb{C}_\delta$.

The simplest summability procedure corresponds to the case when $\mathcal{E}(\Lambda)$ forms an unconditional basis in $L^2(-\pi,\pi)$. Such systems are characterized by B. Pavlov (\cite{Pav}, see also \cite{LS} for another proof).

\begingroup
\setcounter{tmp}{\value{theorem}}
\setcounter{theorem}{0} 
\renewcommand\thetheorem{\Alph{theorem}}
\begin{theorem}
Let $\Lambda\subset\mathbb{C}_\delta$, $\delta>0$. In order that the system $\mathcal{E}(\Lambda)$ be an unconditional basis in $L^2(-\pi,\pi)$ it is necessary and sufficient that
\begin{itemize}
\begin{item}
$|G(x)|^2$ satisfies the Muckenhoupt condition $(A_2)$:
\begin{equation}
\sup_{I =[a,b]}\frac{1}{|I|^2}\int_I|G(x)|^2dx\int_I|G(x)|^{-2}dx<\infty, 
\label{M}
\end{equation}
\end{item}
\begin{item}
The set $\Lambda$ satisfies the Carleson  condition $(C)$:
\begin{equation}
\sup_{\lambda\in\Lambda}\sum_{\mu\in\Lambda,\mu\neq\lambda}\frac{(1+|\Im \lambda|)(1+|\Im\mu|)}{|\lambda-\mu|^2}<\infty.
\label{C}
\end{equation}
\end{item}
\end{itemize}
\end{theorem}
\endgroup
\setcounter{theorem}{\thetmp} 

This theorem was proved by Pavlov under the additional restriction $\sup_{\lambda\in\Lambda} |\Im\lambda|<\infty$ and by Nikolski, see
\cite{HNP}, assuming only $\inf_{\lambda\in\Lambda}\Im\lambda>-\infty$. Finally, Minkin \cite{Min} got rid of all assumptions on $\Lambda$.

One of the motivations for our article is to trace "collaboration" between the two conditions. It seems that the $(A_2)$ condition provides existence of some linear summation method, 
while the choice of this method depends upon "how far"  the sequence $\Lambda$ is from  condition $(C)$.

It is almost straightforward 
that if the generating function $G$ is of exponential type $\pi$ in both halfplanes $\mathbb{C}_\pm$ 
and also $|G(x)|^2\in(A_2)$, then the system $\mathcal{E}(\Lambda)$ is complete and minimal.

\begin{theorem}
Let   the generating function $G$ be of exponential type $\pi$ in both 
halfplanes $\mathbb{C}_\pm$ and also satisfy the  Muckenhoupt condition \eqref{M}. 
Then $\mathcal{E}(\Lambda)$ admits a linear summation method in $L^2(\pi,\pi)$.
\label{main}
\end{theorem}

\begin{remark} A weaker reconstruction property was proved by Gubreev and Tarasenko in \cite[Theorem 2.5]{Gubr} by using spectral decomposition of model operators in de Branges space. Namely if $|G(x)|^2\in(A_2)$, then 
$$f\in\closspan\{(f,\phi_\lambda)e^{i\lambda t}\}\text{ for any } f\in L^2(-\pi,\pi).$$
\end{remark}

This   follows of course from the summation theorem above, yet in Section \ref{hc} we give a simple explicit proof. The corresponding techniques proved to  be useful in other problems of this kind.

A short historical remark may be pertinent. Complete and minimal systems in (generally speaking) Banach spaces with total biorthogonal system are called strong $M$-basis or hereditarily complete systems if each vector in the space belongs to the subspace spanned by the non-zero elements of its Fourier series.  Existence of linear summation method may be considered as an extreme case of hereditary completeness.
Other examples of herididary complete systems have been considered by Markus in connection with spectral synthesis for linear operators, Katavolos, Lambrou and Papadakis in connection with reflexive algebras. Nikolski, Dovbysh and Sudakov found a parametrization of all nonhereditarily complete systems. We refer the reader to \cite{BBB} and references therein.

\begin{remark}
It is shown recently, see \cite{BBB}, that one cannot omit $(A_2)$ condition, generally speaking: there exist $\Lambda\in(C)$
and  $f\in L^2(-\pi,\pi)$ such that the system $\mathcal{E}(\Lambda)$ is complete and minimal in $L^2(-\pi,\pi)$ and,
\begin{equation}
f\not\in\closspan\{(f,\phi_\lambda)e^{i\lambda t}\}.
\label{nonhered}
\end{equation}
\end{remark}
On the other hand  it was recently proved \footnote{A. Borichev, private communication} that there exists hereditarily complete system $\mathcal{E}(\Lambda)$, $\Lambda\subset\mathbb{C}_\delta$, such that $|G(x)|^2\not\in(A_2)$.
The authors wonder if the following statement (converse to Theorem \ref{main}) is true:

{\it Let a sequence $\Lambda\subset\mathbb{C}_\delta$, $\delta>0$ be such that the system $\mathcal{E}(\Lambda)$ is complete and minimal and the series \eqref{mainS} admits a linear summation method. Then the generating function satisfies the Muckenhoupt condition: $|G(x)|^2\in(A_2).$}

We reformulate our problem in the Paley-Wiener space 
$$\pw:=\{F: F(z)=\frac{1}{2\pi}\int_{-\pi}^{\pi}g(t)e^{itz}dt,\quad g\in L^2(-\pi,\pi)\}.$$
The Fourier transform $\mathcal{F}$ acts unitarily from $L^2(-\pi,\pi)$ onto $\pw$, after this transform the exponential functions become the reproducing kernels in $\pw$,
$$\mathcal{F} \left ( e^{-i\bar{\lambda}t} \right ) =\frac{\sin(\pi(z-\bar{\lambda}))}{\pi(z-\bar{\lambda})}=:k_\lambda(z),$$
while the elements of the biorthogonal system become
$$\mathcal{F}(\phi_\lambda)=:G_\lambda=\frac{G(z)}{G'(\lambda)(z-\lambda)}.$$
Simple duality reasonings show that Theorem \ref{main} is equivalent to the following statement
 
\begin{theorem} Under the hypothesis of Theorem \ref{main} the Lagrange interpolating series 
\begin{equation}
F(z)\sim\sum_{\lambda\in\Lambda}F(\lambda)\frac{G(z)}{G'(\lambda)(z-\lambda)}
\label{interp}
\end{equation}
admits a linear summation method in $\pw$.
\label{main2}
\end{theorem}

\begin{remark}
In Theorems \ref{main}-\ref{main2} the Muckenhoupt condition \eqref{M} can be replaced by the same condition for
the function $G(x+ia)$, $a\in\mathbb{R}$.
\end{remark}
This statement holds since for any $a\in\mathbb{R}$ the norm 
$\|F\|_{(a)}:=\bigl{[}\int_{\mathbb{R}}|G(x+ia)|^2dx\bigr{]}^{\half}$ is an equivalent norm for Paley-Wiener space. In particular one can
 apply 
 Theorems \ref{main}-\ref{main2} for  sequences $\Lambda\subset\mathbb{R}$.

The compactwise summability of series \eqref{mainS} holds under even weaker hypothesis:
\begin{theorem} Let $G$ be an entire function of exponential type $\pi$ in both halfplanes $\mathbb{C}_\pm$ and
\begin{equation}
\int_\mathbb{R}\frac{|G(x)|^2}{1+|x^2|}dx<\infty,\qquad\int_\mathbb{R}\frac{dx}{|G(x)|^2(1+|x|^2)}<\infty.
\label{intG}
\end{equation}
Then there exists a linear summation method $W=\{w(\lambda,n)\}$ with compactwise convergence, namely,
$$\lim_{n\rightarrow\infty}\sum_{\lambda\in\Lambda}w(\lambda,n)F(\lambda)\frac{G(z)}{G'(\lambda)(z-\lambda)}=F(z),\text{ uniformly for } z\in K$$
for all $F\in\pw$ and $K\Subset \mathbb{C}$.
\label{pwise}
\end{theorem}
The conditions \eqref{intG} are weaker than the Muckenhoupt condition \eqref{M}.
 But they also imply that $\mathcal{E}(\Lambda)$ is complete and minimal  in $L^2(-\pi,\pi)$.

We give two explicit constructions of the corresponding summation methods.  The first one  is adjusted
to concrete choice of $\Lambda$ and reflects "how far" is $\Lambda$ from condition $(C)$. This method is a development of ideas in \cite{Vas}. In order to implement this method we introduced {\it weighted model spaces} which may be of independent interest.

The second method stems from Abel--Poisson and Ces\'{a}ro methods.  It also can be used for the proof of Theorem \ref{pwise}. The construction of this method is universal for all exponential systems under consideration.

{\bf Structure of the article.} 
Theorems \ref{main}-\ref{main2} are proved in Section \ref{s4}. The preliminary facts for  are given in Section \ref{s3}. The universal summation method for the series \eqref{mainS}, \eqref{interp} is given in Section \ref{s2}.

Given positive quantities $U(x)$, $V(x)$, the notation $U(x)\lesssim V(x)$ (or, equivalently,
$V(x)\gtrsim U(x)$) means that there is a constant $C$ such that
$U(x)\leq CV(x)$ holds for all $x$ in the set in question. We write $U(x)\simeq V(x)$ if both $U(x)\lesssim V(x)$ and
$V(x)\lesssim U(x)$. 


\section{Weighted model subspaces \label{s3}}

Assume that $\Lambda\subset\mathbb{C}_+$. Consider the Blaschke products
\begin{equation}
B(z)=\prod_{\la\in \La} \frac {\bar{\la}}\la \frac {z-\la}{z-\bar{\la}}, \ 
   B_n(z)=\prod_{\la\in \La, |\la|<n} \frac {\bar{\la}}\la \frac {z-\la}{z-\bar{\la}}, \ \beta_n(z)= B(z)/B_n(z).
\label{Bneq}
\end{equation}
\begin{theorem} Let $\Lambda\subset\mathbb{C}_+$ and the generating function $G$ be such that $|G(x)|^2\in(A_2)$. Then 
$$\sum_{\lambda}\beta_n(\lambda)(f,\phi_\lambda)e^{i\lambda t} \xrightarrow{L^2(-\pi,\pi)}f,\text{ as } n\rightarrow\infty, \text { for any } f\in L^2(-\pi,\pi),$$
$$\sum_{\lambda}\beta_n(\lambda)F(\lambda)\frac{G(z)}{G'(\lambda)(z-\lambda)} \xrightarrow{\pw}F,\text{ as } n\rightarrow\infty, \text { for any } F\in \pw.$$
\label{prth}
\end{theorem}

In this section we introduce the  weighted model subspaces, they will be used for  proving  this theorem.
\subsection{Definition of subspaces}
Given an outer function $\omega(z)$ in $\Cp$,    consider the weighted spaces $L^2(\mathbb{R},|\omega|^{\pm2})$ with the norms
$$
\|f\|^2_\omega  = \int_{-\infty}^\infty |f(x)|^2 |\omega(x)|^2 dx;  \quad  \mbox{and}  \quad
     \|f\|^2_{\omega^{-1}}  = \int_{-\infty}^\infty |f(x)|^2 |\omega(x)|^{-2 }dx,
$$
and the {\it weighted} Hardy spaces
$$
\mathfrak H _\omega^+ = \frac 1 \omega H^2(\mathbb{C}_+), \  \mathfrak H_ {\omega^{-1}}^+ =  \omega H^2(\mathbb{C}_+),\     
\mathfrak H _\omega^- = \frac 1 {\omega^\#} H^2(\mathbb{C}_-), \  \mathfrak H_ {\omega^{-1}}^- =  \omega^\# H^2(\mathbb{C}_-).
$$
Here $H^2(\mathbb{C}_\pm)$ denote the classical Hardy spaces in $\Cpm$ respectively and,
  given a function $f(z)$, $z\in \Cp$ we set
$$
f^\#(z)= \overline{f(\bar {z})}, \ z \in \Cm.
$$ 
The spaces $\mathfrak H _\omega^\pm$ (respectively $\mathfrak H_ {\omega^{-1}}^\pm$) are endowed with the norms $\|\cdot\|_\omega$
 (respectively $\|\cdot\|_{\omega^{-1}}$). 
In what follows we do not distinguish functions holomorphic in $ \mathbb{C}_+$ or $ \mathbb{C}_-$ and their boundary values on $\mathbb{R}$. The projectors
$\mathcal P_{\pm}:L^2(\mathbb{R})\to H^2(\mathbb{C}_{\pm})$ have the form
$$
\mathcal{P}_{\pm}:f\mapsto \pm\frac{1}{2}f(x)+\frac{1}{\pi i}\int_{\mathbb{R}}\frac{f(\zeta)}{x-\zeta}d\zeta.
$$
If, in addition, $|\omega|^2\in(A_2)$ they are bounded in the spaces $L^2(\mathbb{R},|\omega|^{\pm2})$ and yield expansions of these spaces into direct sums 
$$
L^2(\mathbb{R},|\omega|^2)= \mathfrak H _\omega^+\dotplus\mathfrak H _\omega^-\quad\text{ and }\quad
L^2(\mathbb{R},|\omega|^{-2})= \mathfrak H _{\omega^{-1}}^+\dotplus\mathfrak H _{\omega^{-1}}^-.
$$
Next, the spaces $\mathfrak H _\omega^+,\mathfrak H_{\omega^{-1}}^-$; $\mathfrak H_{\omega^{-1}}^+,\mathfrak H_{\omega}^-$ are mutually conjugated with respect to the coupling

\begin{equation}
\lng f_+, f_-\rng = \int_{-\infty}^\infty f_+(\xi) {f_-(\xi)}d\xi, \quad f_\pm \in L^2(\R, |\omega|^{\pm2}).
\label{coupling}
\end{equation}
i.e. $\left (\mathfrak H _\omega^+\right )^*=\mathfrak H_ {\omega^{-1}}^- $ etc.

\medskip

We will also use the subspaces
\[
K_+(\omega)=\closspan_\omega \left \{ \frac{B(z)}{z-\la} \right \}_{\la\in \La}, \
     M_+(\omega)=B\mathfrak H^+_\omega, \ L_+(\omega)=\closspan_\omega \left \{ \frac 1 {z-\bar{\la}} \right \}_{\la\in \La}.
\]
\[
K_+(\omega^{-1})=\closspan_{\omega^{-1}} \left \{ \frac{B(z)}{z-{\la}} \right \}_{\la\in \La}, \
     M_+(\omega^{-1})=B\mathfrak H^+_{\omega^{-1}}, \ L_+(\omega^{-1})=\closspan_{\omega^{-1}} \left \{ \frac 1 {z-{\la}} \right \}_{\la\in \La}.
\]
\[
K_-(\omega)=\closspan_\omega \left \{ \frac{B^\#(z)}{z-\bar{\la}} \right \}_{\la\in \La}, \
     M_-(\omega)=B^\#\mathfrak H^-_\omega, \ L_-(\omega)=\closspan_\omega \left \{ \frac 1 {z-{\la}} \right \}_{\la\in \La}.
\]
\[
K_+(\omega^{-1})=\closspan_{\omega^{-1}} \left \{ \frac{B^\#(z)}{z-\bar{\la}} \right \}_{\la\in \La}, \
     M_-(\omega^{-1})=B^\#\mathfrak H^-_{\omega^{-1}}, \ L_-(\omega^{-1})=\closspan_{\omega^{-1}} \left \{ \frac 1 {z-{\la}} \right \}_{\la\in \La}.
\]
Here $\closspan_\omega$ and $\closspan_{\omega^{-1}}$ stay for the closure of the linear span  in the   norms $\|\cdot\|_\omega$ and $\|\cdot\|_{\omega^{-1}}$ respectively.


\subsection{Weighted model spaces}
 \begin{lemma}
\label{le:01}
\begin{equation}
 L_+(\omega)=  M_-(\omega^{-1})^\perp\cap\mathfrak H _\omega^+, \ L_+(\omega^{-1})= M_-(\omega)^\perp\cap\mathfrak H _\omega^+, 
\label{Leq}
\end{equation}
\begin{equation}
 L_-(\omega)=M_+(\omega^{-1})^\perp\cap\mathfrak H _\omega^-, \  L_-(\omega^{-1})=M_+(\omega)^\perp\cap\mathfrak H _\omega^-,
\label{Leq2}
\end{equation}
here orthogonality is considered with respect to coupling  \eqref{coupling}.
\end{lemma}
Indeed, the fact that $L_+(\omega)\subset M_-(\omega^{-1})^{\perp}$ follows just from the reproducing kernel property of the Cauchy kernel. On the other hand
$$f\in  M_-(\omega^{-1})^{\perp}\Leftrightarrow f\perp B^\#\mathfrak H_{\omega^{-1}}^-\Leftrightarrow B^\#f\perp\mathfrak H_{\omega^{-1}}^-\Leftrightarrow B^\#f\in \mathfrak H_{\omega}^-.$$
Therefore $f$ can be extended in $\mathbb{C}_-$ as a meromorphic function with poles in $\overline{\Lambda}$ which yields $f\in L_+(\omega)$. Thus we have $L_+(\omega)= M_-(\omega^{-1})^{\perp}$. The rest of relations \eqref{Leq}-\eqref{Leq2} can be proved similarly.

\begin{lemma}
\label{le:02}
\[
 L_+(\omega)=  B \mathfrak H _\omega^- \cap \mathfrak H _\omega^+, \  
 L_+(\omega^{-1})=  B \mathfrak H _{\omega^{-1}}^- \cap \mathfrak H _{\omega^{-1}}^+ , 
 \]
\[
L_-(\omega)=B^\# \mathfrak H_\omega^+ \cap \mathfrak H _\omega^- , \  
L_-(\omega^{-1})=B^\# \mathfrak H _{\omega^{-1}}^+ \cap \mathfrak H_{\omega^{-1}}^-.
\]
\end{lemma} 

\begin{proof}  We restrict ourselves just to the first of these relations. Clearly 
$L_+(\omega)\subset \mathfrak H _\omega^+$. We also have for $f\in\mathfrak H _\omega^+$
\[
f\in L_+(\omega) \Leftrightarrow f \perp B^\#  \mathfrak H_{\omega^{-1}}^- \Leftrightarrow 
         B^\#   f \perp  \mathfrak H_{\omega^{-1}}^-  \Leftrightarrow B^\#   f \in \mathfrak H _\omega^-
 \Leftrightarrow    f \in  B \mathfrak H _\omega^-.
\]
\end{proof}

\begin{lemma}
$$
 K_+(\omega^{-1})=L_+(\omega^{-1}), \  K_+(\omega)=L_+(\omega), \   K_-(\omega)=L_-(\omega), \ K_-(\omega^{-1})=L_-(\omega^{-1}).
$$
\end{lemma}

\begin{proof} Again we restrict ourselves to the first of these relations:
\begin{align*}
K_+(\omega^{-1})=
        \closspan_{\omega^{-1}} \left \{ \frac{B(z)}{z-\la} \right \}_{\la\in \La}=
                  B(z)   \closspan_{\omega^{-1}} \left \{ \frac{1}{z-\la} \right \}_{\la\in \La}  \\=
                             B(z)  L_-(\omega^{-1})= B(z) \left (   B^\# \mathfrak H _{\omega^{-1}}^+ \cap \mathfrak H_{\omega^{-1}}^-    \right ) =   L_+(\omega^{-1})     
\end{align*}

\end{proof}

\subsection{Projection operators}

\begin{lemma}
\label{le:04}
\begin{equation}
K_+(\omega)\dotplus M_+(\omega)=\mathfrak H _\omega^+, \ 
K_-(\omega)\dotplus M_-(\omega)=\mathfrak H _\omega^-, 
\label{pr1}
\end{equation}
 \begin{equation}
    K_+(\omega^{-1})\dotplus M_+(\omega^{-1})=\mathfrak H _{\omega^{-1}}^+, 
\ K_-(\omega^{-1})\dotplus M_-(\omega^{-1})=\mathfrak H _{\omega^{-1}}^-.
\label{pr2}
\end{equation}
\end{lemma}

\begin{proof}
We restrict ourselves to the first relation only. Observe that, if $g\in 
 \mathfrak H _{\omega^{-1}}^-$ and $g\perp (K_+(\omega)\cup M_+(\omega)) $ then $g=0$.
Indeed $g\perp K_+(\omega)$ $\Rightarrow$ $g|_{\bar{\La}}=0$ $\Rightarrow$ $g=B^\#g_1$, 
for some $g_1\in \mathfrak H _{\omega^{-1}}^-$. Further 
$g\perp M_+(\omega)$ $ \Rightarrow$ $\lng B^\#g_1, Bf\rng=0$ for all $f\in \mathfrak H _\omega^+$,
therefore $g_1=0$, hence 
$K_+(\omega)+ M_+(\omega)$ is dense in $\mathfrak H _\omega^+$.

In order to complete the proof it suffices to mention that the operator 
$$
P_+= I-B\mathcal P_+ B^\# : \mathfrak H _\omega^+ \to K_+(\omega)
$$
is the projector onto $K_+(\omega)$ annihilating $M_+(\omega)$. Boundedness of $P_+$ follows from 
the inclusion $|\omega|^2\in (A_2)$.
\end{proof}
\begin{remark}
Representations \eqref{pr1}-\eqref{pr2} are analogs of the classical representation $H^2=K_\Theta\oplus\Theta H^2$,
where $\Theta$ is an inner function and $K_\Theta $ is the standard model space. 
\end{remark}


\section{Projection summation method \label{s4}}

\subsection{Summation in $K_+(\omega)$ }

Given any $f\in K_+(\omega)$ we consider the interpolating series
\[
f \sim \sum_{\la\in \La } f(\la) \frac {B(z)}{ B'(\la) (z-\la)} 
\]
and look for summation method for this series.

We use the method constructed in \cite{Vas} for the classical model spaces. For each $n>0$ we define
$\Lambda_n=\{\lambda\in\Lambda, |\lambda|<n\}$ and let $B_n(z)$, $\beta_n(z)$ be defined in \eqref{Bneq}.

Consider the subspaces
\[
K_{+, n}(\omega)= \closspan_\omega \left \{ \frac 1 {z-\bar{\la}}  \right \}_{\la\in \La_n}=
     \closspan_\omega \left \{ \frac {B_n(z)}{z-\la} \right \}_{\la\in \La_n} \subset \mathfrak H _\omega^+,
\]
\[
 M_ {+,n}(\omega)= B_n  \mathfrak H _\omega^+,  \  K_{-, n}(\omega^{-1})=\closspan_{\omega^{-1}} \left \{ \frac 1 {z-\la} \right \}_{\la\in \La_n}
\]
and the corresponding projection operators $P_{+,n}= I - B_n\mathcal P_+ B^\#_n: 
\mathfrak H _\omega^+ \to K_{+, n}(\omega)$.

We have
\begin{equation}
\label{eq:20}
K_{+, n}(\omega) \nearrow K_{+}(\omega), \ P_{+,n}|_ {K_{+}(\omega) }\xrightarrow{s}  I\bigl{|}_ {K_{+}(\omega) }\text { as } n\rightarrow\infty.
\end{equation}

The biorthogonality relations
\[
\left \lng \underbrace{ \frac 1 {B'_n(\la)} \frac{B_n(z)}{z-\la}}_{\in K_{+,n}(\omega)}, 
                \underbrace{\frac 1 {z-\mu}}_{\in K_{-,n}(\omega^{-1})} \right \rng = \delta_{\la,\mu}, \ \la,\mu \in \La_n
\]
allow one to express $P_{+,n}$ as:
\[
P_{+,n}f =\sum_{\la\in \La_n } \left \lng f, \frac 1 {z-\la} \right \rng \frac {B_n(z)}{ B'_n(\la) (z-\la)} =
        \sum_{\la\in \La_n } f(\la) \frac {B_n(z)}{ B'_n(\la) (z-\la)}
\]
Relation \eqref{eq:20} now takes the form
\[
P_{+,n}f =\frac 1 {\beta_n(z)} \sum_{\la\in \La_n } f(\la) \beta_n(\la)\frac {B(z)}{ B'(\la) (z-\la)}   \to f, \  n\to \infty, 
\ f \in 
K_{+}(\omega).
\]
Since $|\beta_n(x)|=1$ and $\beta_n(x) \to 1$ as $n\to \infty$ we obtain
\[
S_{+,n}f = \sum_{\la\in \La } f(\la) \beta_n(\la)\frac {B(z)}{ B'(\la) (z-\la)}   \to f, \  n\to \infty, 
\ f \in 
K_{+}(\omega).
\]
This is the desired summation method in $K_{+}(\omega)$.


\subsection{Summation in $\pw$}

The generating function $G$ admits the natural factorization
\begin{equation}
\label{eq:30}
G(z)=\begin{cases}
     \omega(z) B(z) e^{-i\pi z}, &\text{if $ z\in \Cp$} \\
     \omega^\#(z) e^{i\pi z}, & \text{if $z\in \Cm$},
        \end{cases}
\end{equation}
here $\omega$ is an outer function in $\Cp$ and $B$ is the Blaschke product 
with $\La$ as the zero set.

Given $F\in PW_\pi$ consider the function 
\begin{equation}
\label{eq:32}
\Phi(z)=\frac{F(z)e^{i\pi z}}{\omega(z)} \in \mathfrak H _\omega^+.
\end{equation}

Then convergence/summability of  the interpolation series \eqref{interp} in $L^2(\R)$ is equivalent to 
convergence/summability  of the series
\[
\Phi(z)\sim \sum_{\la\in \La}  \Phi(\la) \frac{B(z)}{B'(\la)(z-\la)}
\]
in $\|\cdot\|_\omega$ norm. In order to apply the previous results, and thus prove Theorem  \ref{prth}, it suffices to prove that $\Phi\in  K_+(\omega)$. 

Indeed, it follows from \eqref{eq:30} that  
$$
\omega(x)=\omega^\#(x)e^{2i\pi x} B^\#(x), \quad x \in \R,
$$
therefore 
 $$ 
\Phi(x)=\frac {F(x)e^{-i \pi x}}{\omega^\#(x)B^\#(x)}, \ x\in \R.
$$
It is now straightforward that 
  $\Phi \in  M_-(\omega^{-1})^\perp\cap\mathfrak{H}^+_\omega= K_+(\omega)$.


\subsection{The general case} We briefly describe changes to be done in the case when the sequence $\Lambda$ does not belong to the upper(lower) halfplane. From condition $|G(x)|^2\in(A_2)$ we conclude that $\Lambda\cap\mathbb{R}=\emptyset$.
Denote $\Lambda^{\pm}=\Lambda\cap\mathbb{C}_{\pm}$;
$$
B^{\pm}(z)=\prod_{\la\in \La^{\pm}} \frac {\bar{\la}}\la \frac {z-\la}{z-\bar{\la}}, \ 
   B^{\pm}_n(z)=\prod_{\la\in \La^{\pm}, |\la|<n} \frac {\bar{\la}}\la \frac {z-\la}{z-\bar{\la}}, \ \beta^{\pm}_n(z)= B^{\pm}(z)/B^{\pm}_n(z).
$$
\begin{theorem} Let the generating function $G$ be such that $|G(x)|^2\in(A_2)$. Then 
$$\sum_{\lambda\in\Lambda^{+}}\beta^{+}_n(\lambda)(f,\phi_\lambda)e^{i\lambda t} +
\sum_{\lambda\in\Lambda^{-}}\beta^{-}_n(\lambda)(f,\phi_\lambda)e^{i\lambda t} 
\xrightarrow{L^2(-\pi,\pi)}f,$$
$$\text{ as } n\rightarrow\infty, \text { for any } f\in L^2(-\pi,\pi),$$
\begin{equation}
\sum_{\lambda\in\Lambda^{+}}\beta^{+}_n(\lambda)F(\lambda)\frac{G(z)}{G'(\lambda)(z-\lambda)}
+ \sum_{\lambda\in\Lambda^{-}}\beta^{-}_n(\lambda)F(\lambda)\frac{G(z)}{G'(\lambda)(z-\lambda)}\xrightarrow{\pw}F,
\label{Fpm}
\end{equation}
$$\text{ as } n\rightarrow\infty, \text { for any } F\in \pw.$$
\label{prth2}
\end{theorem}
 
The proof goes in a natural way, one has to replace representation \eqref{eq:30}
by 
\begin{equation}
\label{factorization}
G(z)=\begin{cases}
     \omega(z) B^{+}(z) e^{-i\pi z}, &\text{if $ z\in \Cp$} \\
     \omega^\#(z)B^{-}(z) e^{i\pi z}, & \text{if $z\in \Cm$},
        \end{cases}
\end{equation}
and repeat the argument from the previous sections considering the two summands in \eqref{Fpm} separately.


\section{Universal summation method\label{s2}}

We start with assuming that $\Lambda\subset\mathbb{C}_+$. Later in Subsection \ref{ul} we indicate the changes needed in the general case. As in the previous section  one can apply the Fourier transform $\mathcal{F}$ and reset the  problem into the Paley-Wiener space $\pw$. Thus we have to construct a summation method for the series 
$$
\sum_{\lambda\in \Lambda} (F, G_\lambda)k_\lambda, \ \ F\in\pw,
$$
where $(\cdot,\cdot)$ stays for the standard inner product in $L^2(\mathbb{R})$.

Given any $W:=\{w(\lambda,n)\}$ satisfying Definition \ref{summ} we clearly have
$$\sum_{\lambda\in\Lambda}w(\lambda,n)(F,G_\lambda)k_\lambda\rightarrow F$$
for the dense set of finite linear combinations of $\{k_\lambda\}_{\lambda\in\Lambda}$. Therefore in order to prove that $W$ generates a linear summation method it suffices to prove that the operators
\begin{equation}
T_n: F\rightarrow \sum_{\lambda\in\Lambda}w(\lambda,n)(F,G_\lambda)k_\lambda 
\label{ten}
\end{equation}
are uniformly bounded.

We will prove that for an appropriate choice of $W$ the adjoint operators
$$
T^*_n: F\rightarrow \sum_{\lambda\in\Lambda}w(\lambda,n)(F,k_\lambda)G_\lambda=\sum_{\lambda\in\Lambda}w(\lambda,n)F(\lambda)\frac{G(z)}{G'(\lambda)(z-\lambda)}.
$$
 are uniformly bounded.
 
 
\subsection{Summation matrix} We define the summation matrix $W=\{w(\lambda,n)\}$   by choosing  an appropriate sequence
$\{\alpha_n\}$, $\alpha_n>0$ and   increasing subsets $\Lambda_n \subset \Lambda$, $\cup \Lambda_n=\Lambda$.
This choice will be specified later in Section \ref{result}.

We set   
$$
w_n(z)=\exp\biggl{[}-i\alpha_nl_n\int_{|u|>1\slash2}\biggl{[}\frac{1}{z\slash l_n - u}+\frac{1}{u}\biggr{]}du\biggr{]},\quad z\in\mathbb{C}^+.
$$
So, $w_n(z)$ is an outer function with modulus $1$ in the interval $[-l_n\slash2, l_n\slash2]$ and modulus $e^{-\pi\alpha_n l_n}$
in the set $\mathbb{R}\setminus[-l_n\slash2, l_n\slash2]$.
Put
 $$
 w(\lambda, n)=\begin{cases}
 w_n(\lambda),\quad \lambda\in \Lambda_n\\
0, \quad  \qquad \lambda \not\in \Lambda_n.
\end{cases}
$$
The sequences $l_n \rightarrow\infty$ and $\alpha_n \rightarrow0$ will be chosen later. It is easy to verify that for any $\lambda$, $w_n(\lambda)\rightarrow 1$, $n\rightarrow\infty$.

 
 \subsection {Choice of $\Lambda_n$ and representation of the operators  $T_n^*$ }
 
 Given  sequences $\{l_n\}$, $l_n\to \infty$ and $\{c_n\}$, $c_n>0$ consider the contours 
\begin{multline}
R_n=[l_n,ic_nl_n], \quad L_n=[-l_n,ic_nl_n],\quad
C_n=[-l_n,l_n]\cup R_n\cup L_n,\\ 
\label{Gamma}
\end{multline}
and take $\Lambda_n=\Lambda\cap\inter C_n$. Then, for $F\in \pw$, we have   
 \begin{multline}
 \label{testar}
 T^*_nF(x)=G(x)\sum_{\lambda\in\Lambda_n}\frac{w_n(\lambda)F(x)}{G'(\lambda)(x-\lambda)} \\
= \frac{G(x)}{2\pi i}\int_{C_n}\frac{F(\zeta)w_n(\zeta)}{G(\zeta)(\zeta-x)}d\zeta+\frac{1}{2}F(x)w_n(x)\chi_{[-l_n,l_n]}(x)\\
 =\frac{G(x)}{2\pi i}\biggl{[}\int_{[-l_n,l_n]}+\int_{R_n}+\int_{L_n}\biggr{]}\frac{F(\zeta)w_n(\zeta)}{G(\zeta)(\zeta-x)}d\zeta+\frac{1}{2}F(x)w_n(x)\chi_{[-l_n,l_n]}(x)\\
 =I_{1,n}(x)+I_{2,n}(x)+I_{3,n}(x)+\frac{1}{2}F(x)w_n(x)\chi_{[-l_n,l_n]}(x),
 \end{multline}
here $\chi_{[-l_n,l_n]}$ is the indicator  function of $[-l_n,l_n]$.  The $L^2$-norm of each summand in the right hand side  of \eqref{testar}
will be estimated separately.  Clearly $L^2$-norm of the last summand is bounded by $\|F\|_2$.

  Estimate of $\|I_{1,n}\|$ is straightforward; it follows from the fact that $w_nFG^{-1}$ belongs to the weighted space 
  $$
  L^2(\mathbb{R}, |G|^2):=\{H: \int_\mathbb{R}|H(x)|^2|G(x)|^2<\infty\}.
  $$ 
  and from the fact that, since $|G|^2\in(A_2)$,  the Hilbert transform is bounded in $L^2(\mathbb{R}, |G|^2)$.

It remains to choose the appropriate $\{l_n\}$ and $\{c_n\}$ and then estimate  $\|I_{2,n}\|$ and $\|I_{3,n}\|$.



\subsection{Inner-outer factorization and choice of contours $C_n$} Let $\omega$ be an outer function in $\mathbb{C}_+$ such that $|\omega(x)|=|G(x)|$, $x\in\mathbb{R}$ and let $B$ be the Blaschke product with $\Lambda$ as the zero set. In order to estimate $I_{2,n}$ and $I_{3,n}$ we use the inner-outer factorization  $$e^{i\pi z}G(z)=c\omega(z)B(z),\quad |c|=1.$$
Using the estimate for the upper density of $\Lambda$ and formula $\arg B'(t)=2\sum_\lambda\frac{\Im \lambda}{|t-\lambda|^2}$ we get
$$|\arg B(x)-\arg B(0)|\leq 2\pi\cdot\#\bigl{[} \Lambda\cap\{|z|<2|x|\}\bigr{]}+o(|x|)\lesssim 1+|x|.$$

\begin{lemma}
Let $B$ be a Blaschke product such that $|\arg B(x)|\lesssim 1+|x|$. Then there exists a sequence $l_n\to \infty$ and $\{c_n\}$, $1\leq c_n\leq 10$ and decreasing function $\varepsilon(t)\rightarrow 0$, $t\rightarrow\infty$ such that for the contours   $C_{n}$ defined in \eqref{Gamma}.
\begin{equation}
-\log|B(\lambda)|\leq \varepsilon(|\lambda|)|\lambda|,\quad \lambda\in\bigcup_n C_{n}.
\label{EpsEst}
\end{equation}
\label{plemma}
\end{lemma}

We postpone proof of this lemma until Subsection \ref{plemm} and complete the construction of universal summation method.
 

\medskip

\subsection{}   {\bf Estimates of $\|I_{2,n}\|$ and $\|I_{3,n}\|$ }  are similar.  We restrict ourselves to   $\|I_{2,n}\|$ only.   Let $\{C_n\}$ be the contours  chosen in Lemma \ref{plemma}. 
 Fix $\alpha_n=1000\varepsilon(l_n\slash 10)$, where $\varepsilon(t)$ is a function from Lemma \ref{plemma}. We have
$$
\|I_{2,n}\|_{L^2(\mathbb{R})}=\sup_{g\in L^2,\|g\|\leq 1}\biggl{|}\underbrace{\int_{\mathbb{R}}g(x)I_{2,n}(x)dx}_J\biggr{|}.
$$
and 
 $$
 J=\frac{1}{2\pi i}\int_{R_n}\frac{F(\zeta)A(\zeta)}{G(\zeta)}w_n(\zeta)d\zeta=\frac{1}{2\pi i} \int_{R_n}e^{-i\pi \zeta}F(\zeta)\cdot\frac{A(\zeta)}{\omega(\zeta)}\cdot\frac{w_n(\zeta)}{B(\zeta)}d\zeta,
 $$
where
$$
A(\zeta)=\int_\mathbb{R}\frac{g(x)G(x)}{x-\zeta}dx.
$$

By using the  inequalities,  
$$\log|w_n(\zeta)|\leq -\alpha_n l_n\int_{|u|>1\slash2}\frac{\Im \zeta\slash l_ndu}{|\zeta\slash l_n-u|^2}\leq -\alpha_nl_n\slash5,$$  
$\zeta\in R_n$ and also Lemma \ref{plemma}
we  get
$$
\frac{|w_n(\zeta)|}{|B(\zeta)|}\leq\frac{e^{-200\varepsilon(l_n\slash10)l_n}}{e^{-\varepsilon(|\zeta|)|\zeta|}}\leq\frac{e^{-\varepsilon(|\zeta|)|\zeta|}}{e^{-\varepsilon(|\zeta|)|\zeta|}}=1,\quad \zeta\in R_n,
$$

and 
$$
|J|^2\lesssim \int_{R_n}|e^{-i\pi\zeta}F(\zeta)|^2|d\zeta|\cdot\int_{R_n}\frac{|A(\zeta)|^2}{|\omega(\zeta)|^2}|d\zeta|.
$$
From the boundedness of Hilbert transform with weight $|\omega|^2$ we get 
$$
\frac{A(\zeta)}{\omega(\zeta)}\in H^2(\mathbb{C}^+) \ \text{and} \  \biggl{\|}\dfrac{A(\zeta)}{\omega(\zeta)}\biggr{\|}_{H^2(\mathbb{C}_+)}\lesssim \|g\|\leq 1.
$$
On the other hand, $|d\zeta|\bigr{|}_{R_n}$ is a Carleson measure for $H^2(\mathbb{C}_+)$.
Finally,
$$
|J|\lesssim\|e^{-i\pi \zeta}F(\zeta)\|_{H^2(\mathbb{C}_+)}\cdot \biggl{\|}\dfrac{A(\zeta)}{\omega(\zeta)}\biggr{\|}_{H^2(\mathbb{C}_+)}\lesssim\|F\|_{\pw}.
$$


\subsection{} {\bf The proof of Theorem \ref{pwise}} can be obtain in a similar way. 
 Here we give only a sketch of the proof.

We have to prove that for any compact $K$ the operators $V_n: \pw\mapsto L^\infty(K)$
$$
V_n: F\rightarrow \sum_{\lambda\in\Lambda}w(\lambda,n)(F,k_\lambda)G_\lambda=\sum_{\lambda\in\Lambda}w(\lambda,n)F(\lambda)\frac{G(z)}{G'(\lambda)(z-\lambda)}.
$$
 are uniformly bounded. We fix $R > 1$ such that $K\subset \{|z|< R\}$ and contours $C_n$. 

For any $x\in K$, we have   
 $$
 V_nF(x)=G(x)\sum_{\lambda\in\Lambda_n}\frac{w_n(\lambda)F(x)}{G'(\lambda)(x-\lambda)} $$
$$= \frac{G(x)}{2\pi i}\biggl{[}\int_{\zeta\in\mathbb{C}^-,|\zeta|=2R}+\int_{[-l_n,l_n]\setminus[-2R,2R]}+ \int_{R_n}+\int_{L_n}\biggr{]}\frac{F(\zeta)w_n(\zeta)}{G(\zeta)(\zeta-x)}d\zeta=I_1+I_2+I_3+I_4.$$
Integrals $I_2, I_3, I_4$ can be easily estimated. On the other hand,
$$I_1\leq2\pi R\sup_{|z|=2R}|F(z)|\cdot\sup_{|z|=2R}|G(z)|^{-1}\lesssim\|F\|.$$

\subsection{\label{ul}}{\bf In the general case   $\Lambda$ has points in both halfplanes $\mathbb{C}_\pm$.} We use the inner-outer factorization \eqref{factorization} and then  choose the  contours $C^-_n$ for the points in the lower halfplane $\mathbb{C}_-$ similarly by using the Blaschke product $B^-$. Afterward we  construct the  operators 
$T^-_n$ and let  $T_n=T^+_n+T^{-}_n$. 
The same reasonings prove that  the   operators $T^-_n$ are  uniformly bounded and so $T_n$ is uniformly bounded and $T_nF\rightarrow F$ for a dense set of $F\in \pw$.$\qed$

 
\subsection{Proof of Lemma \ref{plemma}. Preliminary estimates} 
\label{subsec}
We need the following estimate:
\begin{lemma}
Let $B$ be a Blaschke product such that $|\arg B(x)|\lesssim 1+|x|$. Then there exists a sequence $l_n\to \infty$ and $\{c_n\}$, $1\leq c_n\leq 10$ and constant $C = C(B)$ such that for the contours   $C_{n}$ defined in \eqref{Gamma}.
\begin{equation}
-\log|B(\lambda)|\leq C_B\Im\lambda + C_B,\quad \lambda\in\bigcup_n C_{n}.
\label{Best2}
\end{equation}
Moreover, this inequality holds for all $c_n\in U_n$ , where $U_n\subset[1,10]$ is some set with $|U_n|>8$.
\label{BLemma2}
\end{lemma}
\noindent{ \it Proof}
Let $\{z_l\}$ be the zero set of Blaschke product $B$, $z_l=x_l+iy_l$. We choose the points $\{l_n\}_{n=1}^\infty\subset\mathbb{R}$ so that $\sup_n(\arg B)'(\pm l_n)\lesssim 1$. Hence,
$$
\sup_n(\arg B)'(\pm l_n)=\sup_n\sum_{l}\frac{2y_l}{(\pm l_n-x_l)^2+y^2_l}<\infty,
$$
and the sequences $\Lambda\pm l_n$ uniformly satisfy the Blaschke condition. 

Hayman theorem (see, \cite[Lecture 15, Theorem 1]{Lev}) states that $|\log|B(z)||=o(|z|)$ for all $z\in\mathbb{C}_+$ except some set of finite (small) view.  We need  the following version of this result.
\begin{proposition}
For any $\varepsilon>0$ there exists a family of disks $D_m(w_m,r_m):=\{|z-w_m|<r_m\}$ such that 
$$\sum_m\frac{r_m}{|w_m|}<\varepsilon,\qquad |\log|B(z)||\leq c|z|,\quad z\not\in\cup_mD_m(w_m, r_m),$$
where $c$ depends only on $\varepsilon$ and the sum of the Blashke series $\sum_{n}\dfrac{y_n}{x^2_n+y^2_n}$. 
\label{hayman-type}
\end{proposition}
In comparison with the original Hayman theorem we require a weaker estimate for $\log|B(z)|$; instead we demand  uniformity of the constant $c$.
The proof of Proposition \ref{hayman-type} can be obtained in the same way as in \cite{Lev}.

We apply Proposition \ref{hayman-type} to $B(z\pm l_n)$: there exists $c_n\in[1,10]$  so that,
$$|B(z\pm l_n)|\gtrsim e^{-c|z|},\quad z\in\{\mp t+ ic_n t, t>0\}.$$
If  $z\in\{\mp t+ ic_n t, t>0\}$, then $|z|\simeq \Im z$ and we get the required estimate for $B$ and contours $C_{n}$ as in \eqref{Gamma}. $\qed$

\subsection{Proof of Lemma \ref{plemma}\label{plemm}}
We choose some contours $C_{n}$ from Lemma \ref{BLemma2}.
Now we apply Hayman theorem for the Blashke product $B$. So there exists a sequense of disks $D_n=\{z: |z-z_n|<r_n\}$
of finite view 
\begin{equation}
\sum_n\frac{r_n}{|z_n|}<\frac{1}{1000},
\label{view}
\end{equation}
such that $-\log|B(z)|\leq \varepsilon_1(|\lambda|)|\lambda|$,
 $\lambda\notin \cup_n D_n$, where $\varepsilon_1$ is some function, $\varepsilon_1(t)\rightarrow0$.

Put
$$\Delta_n:=\conv\{l_n,il_n,10il_n\}.$$

Let us consider $\sigma_n = \sum_{k\in \mathcal{N}_n}\frac{r_k}{|z_k|}$, $\mathcal{N}_n$ is a set of all indices $s$ such that
$D_s\cap \Delta_n\neq\emptyset$. If $l_{n+1}>100 l_n$, then $\mathcal{N}_n\cap\mathcal{N}_m=\emptyset$, $n\neq m$.
So, $\sum_n\sigma_n<\frac{1}{1000}$.

We fix a sequence $q_n$ increasing to infinity such that $\sum_nq_n\sigma_n<1\slash 1000$. 

Put 
$$\Delta^+_n=\{z\in\Delta_n: q_n\Im z \geq |\Re z|\}.$$
If for some disk $D_k$, $D_k\cap\Delta^+_n\neq\emptyset$, then
$$\frac{r_k}{|z_k-l_n|}\leq 10q_n\frac{r_k}{|z_k|}.$$

So, $\sum_{k: D_k\cap \Delta^+_n\neq \emptyset}\frac{r_k}{|z_k-l_n|}<\frac{1}{100}.$
Hence, there exists contours $C_n$ as in Lemma \ref{plemma} such that $C_n\cap D_k=\emptyset$ for any $D_k$
such that $D_k\cap\Delta^+_n\neq\emptyset$ and, hence, we have estimate
$$-\log|B(z)|\leq \varepsilon_1(|\lambda|)|\lambda|,\quad \lambda\in C_n\cap\Delta^+_n.$$ 
On the other hand from Lemma \ref{plemma} we get
$$-\log|B(\lambda)|\leq C_B\Im\lambda+C_B\leq5\frac{C_B}{q_n}|\lambda|, \quad\lambda\in C_n\setminus\Delta^+_n.$$
\qed


\subsection{\label{result}}{\bf We collect all pieces together} in order to formulate the final result.  
\begin{theorem} Let $C^\pm_n(=C^\pm_n(-l_n,l_n,\pm ic_nl_n))$ be the set of triangle contours from Lemma \ref{plemma} and 
$B^\pm$ be the Blashke products with zero sets $\Lambda\cap \mathbb{C}^\pm$ respectively. Then the matrix
$$
w(n,\lambda)=\begin{cases}w^\pm_n(\lambda),\quad\lambda\in\mathbb{C}^\pm,\quad\lambda\in\inter C^\pm_n\\
0,\text{ otherwise }\end{cases}
$$
generates a linear summation method for the series \eqref{mainS}, \eqref{interp}.
\end{theorem}

It worth to be mentioned  that   the sequence of contours $C_n$ needs not be  sparse, if, say, 
 $\Lambda\subset\mathbb{C}^+$ and the corresponding Blaschke product $B$ satisfies $(\arg B)'(x)\lesssim1$, then the points $l_n $ can be chosen to be an arbitrary sequence tending to infinity.


\subsection{Hereditary completeness of $\mathcal{E}(\Lambda)$\label{hc}} Here we give a direct proof of the 
hereditary completeness of $\mathcal{E}(\Lambda)$ under the assumptions of Theorem \ref{main}.  The idea of this proof comes from \cite{BBB}.

Assume that there exists $f$ as in \eqref{nonhered}.  One can easily see that in this case there exists a partition $\Lambda=\Lambda_1\cup\Lambda_2$ such that 
$$
{\rm span}\{\{\phi_\lambda\}_{\lambda\in \Lambda_1} \cup \{e^{i\lambda t}\}_ {\lambda\in \Lambda_2}\}
$$
is not dense in $L^2(-\pi,\pi)$ so there exists a non-zero $h\in L^2(-\pi,\pi)$ such that  
  $$
  (h, e^{i\lambda t})(\phi_\lambda, h)=0,\  \lambda\in\Lambda.
  $$
This can be rewritten in the  Paley-Wiener space. Let  $H=\mathcal{F}h$ then   
 $$
 (H,k_\lambda)(G_\lambda, H)=0, \text{ for any } \lambda\in\Lambda.
 $$
We  use the Shannon-Kotelnikov-Whittaker formula  
$$
H(z)=\sum_{n\in\mathbb{Z}}H(n)\frac{\sin(\pi(z-n))}{\pi(z-n)}=\sum_{n\in\mathbb{Z}}H(n)k_n(z).
$$
So, $(G_\lambda, H)=\frac{1}{G'(\lambda)}\sum_{n\in\mathbb{Z}}\frac{G(n)\overline{H(n)}}{\lambda-n}$. 
Consider the meromorphic function 
$$
H(z)\sum_n\frac{\overline{H(n)}G(n)}{z-n};
$$ 
this function vanishes on $\Lambda$. So,
\begin{equation}
H(z)\cdot\sin(\pi z)\sum_n\frac{\overline{H(n)}G(n)}{z-n}=G(z)S(z)
\label{prod}
\end{equation}
for some entire function $S$.
It is known (see, \cite[Lemma 2.2]{BBB}) that, for appropriate $c_0$
$$S(z)=\sin(\pi z)\sum_n\frac{|H(n)|^2}{z-n}+c_0\sin(\pi z).$$
Consider the measure  $\mu_0=\sum_{n\in\mathbb{Z}}\delta_n$. Relation \eqref{prod} takes the form
$$
H(x)\int_\mathbb{R}\frac{\overline{H(t)}G(t)d\mu_0(t)}{x-t}=G(x)\biggl{(}c_0+\int_\mathbb{R}\frac{|H(t)|^2d\mu_0(t)}{x-t}\biggr{)}.
$$
 A similar relation holds for any measure   $\mu_\alpha=\sum_{n\in\mathbb{Z}}\delta_{n+\alpha}$, $\alpha\in[0,1)$, and, hence, the analogous equation is true for the Lebesque measure $dx = \int_0^1\mu_\alpha d\alpha$, $\alpha\in (0,1)$ so integration in
 $\alpha\in (0,1)$ gives
 $$
 \frac{H(x)}{G(x)}\int_\mathbb{R}\frac{\overline{H(t)}G(t)dt}{x-t}= c+\int_\mathbb{R}\frac{|H(t)|^2dt}{x-t}.
 $$
Using the boundedness of Hilbert transform with the weight $|G(x)|^2$ we see that left hand side belongs to $L^1(\mathbb{R} )$. 
In case $c=0$ this immediately yields us to contradiction.  If $c=0$ we use a well known fact:
$$
m\bigl{\{}x:\bigl{|}\int_\mathbb{R}\frac{|H(t)|^2dt}{x-t}\bigr{|}>s\bigr{\}}\gtrsim \frac{1}{s}, \ s\rightarrow 0,
$$
 here $m$ is Lebesgue measure, and  thus  the function $c+\int_\mathbb{R}\frac{|H(t)|^2dt}{x-t}$ is not summable even if $c=0$. We arrive to a contradiction.

\end{document}